\documentclass[reqno,12pt]{amsart}
\usepackage{amssymb}
\usepackage{mathrsfs}
\usepackage{txfonts}
\usepackage{amsfonts}
\usepackage{amstext}
\usepackage{amssymb}
\usepackage{amsmath}
\usepackage{graphicx}
\usepackage{hyperref}
\usepackage{url}

\hypersetup{
        colorlinks   = true,
        citecolor    = blue,
        linkcolor    = blue,
        urlcolor     = blue
}
\allowdisplaybreaks
\newtheorem{theorem}{Theorem}[section]
\newtheorem{proposition}[theorem]{Proposition}
\newtheorem{lemma}[theorem]{Lemma}

\newtheorem{corollary}[theorem]{Corollary}
\newtheorem{remark}[theorem]{Remark}
\renewcommand{\theequation}{\thesection.\arabic{equation}}

\newenvironment{acknowledgement}{\smallskip{\sc Acknowledgement.}\rm}{\smallskip}

\numberwithin{equation}{section}
\newcounter{counterConstant}

\input{tcilatex}

\def\RM{\rm}

\def\diint{\mathop{\int\int}}

\def\dint{\displaystyle\int}

\def\Xint#1{\mathchoice
{\XXint\displaystyle\textstyle{#1}}%
{\XXint\textstyle\scriptstyle{#1}}%
{\XXint\scriptstyle\scriptscriptstyle{#1}}%
{\XXint\scriptscriptstyle\scriptscriptstyle{#1}}%
\!\int}
\def\XXint#1#2#3{{\setbox0=\hbox{$#1{#2#3}{\int}$ }
\vcenter{\hbox{$#2#3$ }}\kern-.6\wd0}}

\def\oint{\Xint-}

\def\FRAME#1#2#3#4#5#6#7#8
{
 \begin{figure}[ptbh]
 \begin{center}
 \includegraphics[height=#3]{#7}
 \caption{#5}
 \label{#6}
 \end{center}
 \end{figure}
}

\def\enddoc{\end{document}}

\begin{document}
\title[heat kernel upper bound]{Davies method for heat kernel upper bounds
of non-local Dirichlet forms on ultra-metric spaces}
\author[Gao]{Jin Gao}
\address{Department of Mathematical Sciences, Tsinghua University, Beijing
100084, China.}
\email{gao-j17@mails.tsinghua.edu.cn}
\thanks{\noindent Supported by NSFC No.11871296.}
\date{March 2019}

\begin{abstract}
We apply the Davies method to give a quick proof for upper estimate of the
heat kernel for the non-local Dirichlet form on the ultra-metric space. The
key observation is that the heat kernel of the truncated Dirichlet form
vanishes when two spatial points are separated by any ball of radius larger
than the truncated range. This new phenomenon arises from the ultra-metric
property of the space.
\end{abstract}

\subjclass[2010]{35K08, 28A80, 60J35}
\keywords{Heat kernel, ultra-metric, Davies method.}
\maketitle


\section{Introduction}

We are concerned with the heat kernel estimate for the non-local Dirichlet
form on the ultrametric space. Let $\left( M,d,\mu \right) $ be an \emph{%
ultra-metric measure} space, that is, $M$ is locally compact and separable, $%
d $ is an ultra-metric, and $\mu $ is a Radon measure with full support in $%
M $.

Recall that a metric $d$ is called an \emph{ultra-metric} if for any points $%
x,y,z\in M$,%
\begin{equation*}
d(x,y)\leq \max \{d(x,z),d(z,y)\}.
\end{equation*}%
For any $x\in M$ and $r>0$ the \emph{metric} ball $B(x,r)$ is defined by
\begin{equation*}
B(x,r):=\{y\in M,d(x,y)\leq r\}.
\end{equation*}%
It is known that any two metric balls are either \emph{disjoint} or one
\emph{contains} the other (see \cite{Bendikov2014Isotropic,BendikovGAHU2018})
. Thus any ball is both closed and open so that its boundary is empty.

An ultra-metric space is totally disconnected. So the process described by
the heat kernel is a pure jump process. In this paper, { we consider the
the Dirichlet form $\mathcal{E}$ only having non-local part without
killing part in the Beurling-Deny decomposition (see \cite[p.120]%
{FukushimaOshimaTakeda.2011.489})}. Let $\mathcal{E}$ be the energy form given by
\begin{equation}
\mathcal{E}(f,g)=\diint\limits_{M\times M\setminus
{ \text{diag}}}(f(x)-f(y))(g(x)-g(y))dj(x,y),  \label{EE}
\end{equation}%
where $j$ is a {  symmetric} Radon measure on $(M\times M)\setminus${ diag}. For simplicity, we denote $\mathcal{E}(f,f)$ by $\mathcal{E}(f)$. We need to specify the
domain of $\mathcal{E}$. Let $\mathcal{D}$ be the space defined by%
\begin{equation}
\mathcal{D}=\left\{ \sum_{i=0}^{n}c_{i}\mathbf{1}_{B_{i}}:n\geq 1,c_{i}\in
\mathbb{R}
, B_{i}\text{ are compact disjoint balls}\right\} .  \label{dd}
\end{equation}%
Let the \emph{space} $\mathcal{F}$ be determined by
\begin{equation}
\mathcal{F}\ =\text{ the \emph{closure }of }\mathcal{D}\text{ under norm}%
\sqrt{\mathcal{E}(u)+\left\Vert u\right\Vert _{2}^{2}}.  \label{FF}
\end{equation}%
Then $\left( \mathcal{E},\mathcal{F}\right) $ is a \emph{regular} Dirichlet
form in $L^{2}:=L^{2}(M,\mu )$ if the measure $j$ satisfies that $%
j(B,B^{c})<\infty $ for any ball $B$ {  (see \cite[Theorem 2.2]%
{BendikovGAHU2018})}.

Recall that the indicator function $\mathbf{1}_{B}$ for any ball $B$ belongs
to the space $\mathcal{F}$ if $j(B,B^{c})<\infty $, since
\begin{equation*}
\mathcal{E}(\mathbf{1}_{B})=2j(B,B^{c})<\infty \text{,}
\end{equation*}%
{ (see \cite[formula (4.1)]{BendikovGAHU2018})}.

In the sequel, we will fix some numbers $\alpha >0,\beta >0$, and some value
{ $R_{0}\in(0, \text{diam}~M]$, which contains the case when $R_0 =\text{diam}~M=\infty$}. The letter $C$ is universal positive constant which may vary at
each occurrence.

We list some conditions to be used later on.

\begin{itemize}

\item \textbf{Condition }$\left( TJ\right) $: there exists a transition
function $J(x,dy)$ such that
\begin{equation*}
dj(x,y)=J(x,dy)d\mu (x),
\end{equation*}%
and for any ball $B:=B(x,r)$ with $x\in M$ and $r\in (0,R_{0}),$
\begin{equation*}
J(x,B^{c})\leq \frac{C}{r^{\beta }}.
\end{equation*}

\item \textbf{Condition }$(DUE)$: the heat kernel $p_{t}(x,y)$ exists and
satisfies
\begin{equation*}
p_{t}(x,y)\leq \frac{C}{t^{\alpha /\beta }}
\end{equation*}%
for all $t\in (0,R_{0}^{\beta })$ and $\mu $-almost all $x$, $y\in M$.

\item \textbf{Condition }$(wUE)$\textbf{:} the heat kernel $p_{t}$ exists
and satisfies%
\begin{equation*}
p_{t}(x,y)\leq \frac{C}{t^{\alpha /\beta }}\left( 1+\frac{d(x,y)\wedge R_{0}%
}{t^{1/\beta }}\right) ^{-\beta }
\end{equation*}%
for all $t\in (0,R_{0}^{\beta })$ and $\mu $-almost all $x$, $y\in M$.
\end{itemize}

Our aim in this paper is to show the following implication.

\begin{theorem}
\label{T1} Let $\left( \mathcal{E},\mathcal{F}\right) $ be given by (\ref{EE}%
), (\ref{FF}) on an ultra-metric measure space. Then%
\begin{equation}
(TJ)+(DUE)\Rightarrow (wUE).  \label{1.20}
\end{equation}
\end{theorem}

The result (\ref{1.20}) is not new  (see \cite[Lemma $5.2$, Theorem $12.1$ and { Subsection} $12.3$]{BendikovGAHU2018}), and was proved in \cite[Theorem $2.8$]{BendikovGAHU2018} by using a very complicated idea that was developed in
several papers \cite{GrigoryanHuHu.2017.JFA3311, GrigoryanHuHu.2018,
GrigoryanHuLau.2010.JFA2613, GrigoryanHuLau.2014.TAMS6397} (
After this paper has been finished, {  we were aware that a much simpler proof
was presented in \cite{BendikovGAHU2018} recently)}.
What is new in this paper is that we apply the Davies method developed in
\cite{Davies.1987.AJM319, CarlenKusuokaStroock.1987.AIHPPS245} (see also
\cite{MuruganSaloff.2015.dav}, \cite[p1152]{HuLi.2018} { on general metric
measure spaces including fractals}) to give a much simpler proof for the
implication (\ref{1.20}). The key observation here is that the heat kernel
for the truncated Dirichlet form vanishes at any time when points are
separated by a ball of radius larger than the truncated range { (see Lemma \ref%
{L1} below)}. This new phenomenon arises from the ultra-metric property of the
space $M$ ({ General metric spaces} do not admit this nice property).
It would be interesting to generalize
Theorem \ref{T1} by using the Davies method to the  more general case when the condition $\left(DUE\right)$  becomes
\begin{align}
p_t(x,y)\leq \frac{C}{V(x,\phi^{-1}(t))},\notag
\end{align}
where $\phi$ is an increasing function on $[0,\infty)$, and { $V(x,r)=\mu(B(x,r))$ }is
the volume of the ball $B(x,r)$ that is also sensitive to the center $x$.

\begin{acknowledgement}
The author thanks Jiaxin Hu, Alexander Grigor'yan for suggesting this topic
and Eryan Hu for discussions.
\end{acknowledgement}

\section{Heat kernel for the truncated Dirichlet form}

We need to consider the truncated form and then estimate the corresponding
heat kernel. We will do this by using the Davies method.

For any $\rho >0,$ let $\mathcal{E}^{(\rho)}$ be defined by
\begin{equation}
\mathcal{E}^{(\rho)}(u,v)=\int_{M}\int_{B(x,\rho
)}(u(x)-u(y))(v(x)-v(y))dj(x,y).  \label{eq:DF_q}
\end{equation}%
It is known that if
\begin{equation}
{ \sup_{x\in M}\int_{B(x,r)^{c}}J(x,dy)<\infty,  ~\quad~~~~~~r\in (0,R_{0})}\label{31}
\end{equation}%
then $\left(\mathcal{E}^{(\rho )},\mathcal{F}\right)$ is closable and its {  closure} form
$\left(\mathcal{E}^{(\rho )},\mathcal{F}^{(\rho )}\right)$ is a \emph{regular} Dirichlet
form { (see \cite[Proposition 4.2]{GrigoryanHuLau.2014.TAMS6397})}. Condition $
\left( TJ\right) $ implies condition (\ref{31}), then condition $\left( TJ\right)$
also implies that the form $(\mathcal{E}^{(\rho )},\mathcal{F}^{(\rho )})$ is
regular { (see \cite[Theorem 2.2]{BendikovGAHU2018})}.

\begin{proposition}
If conditions $\left( TJ\right) ,\left( DUE\right) $ hold, then the
following functional inequality called Nash inequality holds: there exists a
constant { $C_N>0$ (subscript $N$ means this constant comes from \lq\lq Nash" inequality)} such that for any $u\in \mathcal{F}\cap L^{1}$,
\begin{equation}
\Vert u\Vert _{2}^{2(1+\nu )}\leq  { C_N} \left( \mathcal{E}^{(\rho)}(u)+K_{0}\Vert
u\Vert _{2}^{2}\right) \left\Vert u\right\Vert _{1}^{2\nu },  \label{Nash}
\end{equation}%
where $\nu :={\beta }/{\alpha }$ and $K_{0}:=\rho ^{-\beta }+R_{0}^{-\beta }$%
.
\end{proposition}

\begin{proof}
Since $\left( DUE\right) $ holds, we have by \cite[Theorem (2.1)]%
{CarlenKusuokaStroock.1987.AIHPPS245} that
\begin{equation}
\Vert u\Vert _{2}^{2(1+\nu )}\leq C\left( \mathcal{E}(u)+R_{0}^{-\beta
}\Vert u\Vert _{2}^{2}\right) \left\Vert u\right\Vert _{1}^{2\nu }.
\label{3.1}
\end{equation}%
On the other hand, { by condition $\left( TJ\right)$ and the
symmetry of $j$, we have},%
\begin{eqnarray*}
\mathcal{E}(u)-\mathcal{E}^{(\rho)}(u)
&=&\int_{M}\int_{M}(u(x)-u(y))^{2}1_{\{d(x,y)>\rho \}}dj(x,y) \\
&\leq &\int_{M}\int_{M}4u^{2}(x)1_{\{d(x,y)>\rho \}}dj(x,y)\leq
4\int_{M}u^{2}(x)d\mu (x)\sup_{x\in M}\int_{B(x,\rho )^{c}}J(x,dy) \\
&\leq &C\rho ^{-\beta }\Vert u\Vert _{2}^{2}.
\end{eqnarray*}%
({ Note} that $\mathcal{F}^{(\rho )}=\mathcal{F}$ by this inequality). Plugging
this into (\ref{3.1}), we obtain (\ref{Nash}).
\end{proof}

We need the following.

\begin{proposition}
\label{P1}Let $B$ be any ball of radius $r>0$. If $0<\rho \leq r$, then
\begin{equation}
\mathcal{E}^{(\rho)}(e^{-\psi }f,e^{\psi }g)=\mathcal{E}^{(\rho)}(f,g)
\label{11}
\end{equation}%
for any $f,g\in \mathcal{F}^{(\rho )}\cap L^{\infty }$, where $\psi =\lambda
\mathbf{1}_{B}$ with $\lambda \in
\mathbb{R}
$.
\end{proposition}

\begin{proof}
Since $\psi =\lambda \mathbf{1}_{B}\in \mathcal{F}$, we see that $e^{\psi
}-1\in \mathcal{F}^{(\rho )}\cap L^{\infty }$ by using the Markov property
of $(\mathcal{E}^{(\rho )},\mathcal{F}^{(\rho )})$, and hence both functions
$e^{\psi }g$ and $e^{-\psi }g$ belong to $\mathcal{F}^{(\rho )}\cap
L^{\infty }$ if $g\in \mathcal{F}^{(\rho )}\cap L^{\infty }$.

Since $0<\rho \leq r$, we see that $B(x,\rho )\subset B(x,r)=B$ if $x\in B$,
whereas $B(x,\rho )\subset B^{c}$ if $x\in B^{c}$ because two balls $%
B(x,\rho )$ and $B$ are disjoint by using the ultra-metric property. Thus%
\begin{eqnarray*}
\psi (x) &=&\psi (y)=\lambda \text{ \ if }x\in B,y\in B(x,\rho ), \\
\psi (x) &=&\psi (y)=0\text{ \ if }x\in B^{c},y\in B(x,\rho ).
\end{eqnarray*}%
It follows that%
\begin{eqnarray*}
\mathcal{E}^{(\rho)}(e^{-\psi }f,e^{\psi }g) &=&\int_{M}\dint_{B(x,\rho
)}(e^{-\psi (x)}f(x)-e^{-\psi (y)}f(y))(e^{\psi (x)}g(x)-e^{\psi (y)}g(y))dj
\\
&=&\int_{B}\dint_{B(x,\rho )}+\int_{B^{c}}\dint_{B(x,\rho )}\cdots \\
&=&\int_{B}\dint_{B(x,\rho )}(e^{-\lambda }f(x)-e^{-\lambda
}f(y))(e^{\lambda }g(x)-e^{\lambda }g(y))dj \\
&&+\int_{B^{c}}\dint_{B(x,\rho )}(f(x)-f(y))(g(x)-g(y))dj \\
&=&\int_{M}\dint_{B(x,\rho )}(f(x)-f(y))(g(x)-g(y))dj=\mathcal{E}^{(\rho)}(f,g),
\end{eqnarray*}%
thus proving (\ref{11}).
\end{proof}

\begin{corollary}
\label{C1}Let $B$ be any ball of radius $r>0$. If $0<\rho \leq r$, then
\begin{equation}
\mathcal{E}^{(\rho)}(e^{-\psi }f,e^{\psi }f^{2p-1})\geq \frac{1}{p}\mathcal{E}^{(\rho)}(f^{p})  \label{12}
\end{equation}%
for any $f\in \mathcal{F}^{(\rho )}\cap L^{\infty }$ and any $p\geq 1$,
where $\psi =\lambda \mathbf{1}_{B}$ with $\lambda \in
\mathbb{R}
$.
\end{corollary}

\begin{proof}
Using the elementary inequality%
\begin{equation*}
(a-b)(a^{2p-1}-b^{2p-1})\geq \frac{2p-1}{p^{2}}(a^{p}-b^{p})^{2}\geq \frac{1%
}{p}(a^{p}-b^{p})^{2}
\end{equation*}%
for any non-negative numbers $a,b$ and any $p\geq 1$, and taking $%
g=f^{2p-1}$ for any $p\geq 1$ in (\ref{11}), we obtain by letting $%
a=f(x),b=f(y)$,%
\begin{eqnarray*}
\mathcal{E}^{(\rho)}(e^{-\psi }f,e^{\psi }f^{2p-1}) &=&\mathcal{E}^{(\rho)}(f,f^{2p-1})
=\dint_{M\times B(x,\rho)}(f(x)-f(y))(f^{2p-1}(x)-f^{2p-1}(y))dj \\
&\geq &\frac{1}{p}\dint_{M\times B(x,\rho )}(f^{p}(x)-f^{p}(y))^{2}dj=\frac{1%
}{p}\mathcal{E}^{(\rho)}(f^{p}),
\end{eqnarray*}%
thus proving (\ref{12}).
\end{proof}

\begin{remark}
\RM Inequality (\ref{12}) is an enhancement of the previous similar results
in \cite[formula (3.11)]{CarlenKusuokaStroock.1987.AIHPPS245}, \cite[formula
(2.8)]{MuruganSaloff.2015.dav} (see also \cite[formula (3.11)]{HuLi.2018})
in the setting of the ultra-metric space.
\end{remark}

For any { $f\in \mathcal{F}\cap L^{\infty }$ with $\|f\|_2=1$}. Denote by $\{Q_{t}\}_{t\geq 0}$ the heat semigroup of $(\mathcal{E}^{(\rho
)},\mathcal{F}^{(\rho )})$. Let%
\begin{equation}
f_{t}:=Q_{t}^{\psi }f:=e^{{ \psi} }Q_{t}(e^{{ -\psi} }f)  \label{ft}
\end{equation}%
be the \emph{perturbed} semigroup\ of $\{Q_{t}\}_{t\geq 0}$.

\begin{proposition}
\label{P2}Let $B$ be any ball of radius $r>0$. If $0<\rho \leq r$, then%
\begin{equation}
\frac{d}{dt}\left\Vert f_{t}\right\Vert _{2p}\leq -\frac{C_{N}^{-1}}{p}%
\left\Vert f_{t}\right\Vert _{2p}^{1+2p\nu }\left\Vert f_{t}\right\Vert
_{p}^{-2p\nu }+\frac{K_{0}}{p}\left\Vert f_{t}\right\Vert _{2p}  \label{13}
\end{equation}%
for any non-negative ${ f_t}\in \mathcal{F}^{(\rho )}\cap L^{\infty }$ and any $%
p\geq 1$, where $\psi =\lambda \mathbf{1}_{B}$ with $\lambda \in
\mathbb{R}
$ as before, and $C_{N}$ is the same as in (\ref{Nash}).
\end{proposition}

\begin{proof}
Note that $f_{t}\in \mathcal{F}^{(\rho )}\cap L^{\infty }$. Then, using (\ref%
{12}) with $f$ being replaced by $f_{t}$,%
\begin{eqnarray}
\frac{d}{dt}\left\Vert f_{t}\right\Vert _{2p}^{2p} &=&2p\int f_{t}^{2p-1}%
\frac{\partial f_{t}}{\partial t}d\mu =2p\left(f_{t}^{2p-1},e^{\psi }\frac{%
\partial }{\partial t}Q_{t}( e^{-\psi }f) \right)  \notag \\
&=&2p\left(e^{\psi }f_{t}^{2p-1},\mathcal{L}Q_{t}( e^{-\psi }f ) \right)=-2p%
\mathcal{E}^{(\rho)}(e^{\psi }f_{t}^{2p-1},e^{-\psi }f_{t})\leq -2\mathcal{E}^{(\rho)}(f_{t}^{p}).  \label{14}
\end{eqnarray}%
On the other hand, applying the Nash inequality (\ref{Nash}) with $u$ being
replaced by $f_{t}^{p}$, we see that%
\begin{equation*}
-\mathcal{E}^{(\rho)}(f_{t}^{p})\leq -C_{N}^{-1}\Vert f_{t}\Vert
_{2p}^{2p(1+\nu )}\left\Vert f_{t}\right\Vert _{p}^{-2p\nu }+K_{0}\Vert
f_{t}\Vert _{2p}^{2p}.
\end{equation*}%
Plugging this into (\ref{14}), we obtain%
\begin{equation*}
2p\left\Vert f_{t}\right\Vert _{2p}^{2p-1}\frac{d}{dt}\left\Vert
f_{t}\right\Vert _{2p}=\frac{d}{dt}\left\Vert f_{t}\right\Vert
_{2p}^{2p}\leq -2C_{N}^{-1}\Vert f_{t}\Vert _{2p}^{2p(1+\nu )}\left\Vert
f_{t}\right\Vert _{p}^{-2p\nu }+2K_{0}\Vert f_{t}\Vert _{2p}^{2p},
\end{equation*}%
which gives (\ref{13}) after dividing $2p\left\Vert f_{t}\right\Vert
_{2p}^{2p-1}$ on the both sides.
\end{proof}

For any integer $k\geq 1$ we define the function $w_{k}(t)$ for $t>0$ by
\begin{equation}
w_{k}(t):=\underset{s\in (0,t]}{\sup }\left\{ s^{(2^{k-1}-1)/(2^{k}\nu
)}\left\Vert f_{s}\right\Vert _{2^{k}}\right\} .  \label{wkt}
\end{equation}%
Clearly, $w_{k}(t)$ is non-decreasing in $t>0$. Note that by (\ref{13}) with
$p=1$,%
\begin{equation*}
\frac{d}{dt}\left\Vert f_{t}\right\Vert _{2}\leq K_{0}\left\Vert
f_{t}\right\Vert _{2}
\end{equation*}%
and thus
\begin{equation*}
\left\Vert f_{t}\right\Vert _{2}\leq \exp (K_{0}t)\left\Vert f\right\Vert
_{2}.
\end{equation*}%
This gives that for any $t>0$%
\begin{equation}
w_{1}(t){ =}\underset{s\in (0,t]}{\sup }\left\{ \left\Vert f_{s}\right\Vert
_{2}\right\} \leq \underset{s\in (0,t]}{\sup }\left\{ \exp
(K_{0}s)\left\Vert f\right\Vert _{2}\right\} \leq \exp (K_{0}t)\left\Vert
f\right\Vert _{2}.  \label{15}
\end{equation}%
We will estimate $w_{k}$ by iteration.

\begin{proposition}
\label{P3}Let $w_{k}$ be defined as in (\ref{wkt}). Then for any
non-negative $f\in \mathcal{F}^{(\rho )}\cap L^{\infty }$%
\begin{equation}
w_{k+1}(t)\leq C_{1}\exp (2K_{0}t)\left\Vert f\right\Vert _{2}  \label{17}
\end{equation}%
for any $t>0$, where $C_{1}$ is some universal constant depending only on $%
\nu ,C_{N}$ (but independent of $t,\rho ,r$ and $\lambda $), and $K_{0}=\rho
^{-\beta }+R_{0}^{-\beta }$ as before.
\end{proposition}

\begin{proof}
For any integer $k\geq 1$, denote by%
\begin{equation*}
u_{k}(s)=\left\Vert f_{s}\right\Vert _{2^{k}}\text{ \ (}s>0\text{)}
\end{equation*}%
for simplicity. Applying (\ref{13}) with $p=2^{k}$, we have
\begin{equation*}
u_{k+1}^{^{\prime }}(s)\leq -\frac{C_{N}^{-1}}{2^{k}}u_{k+1}(s)^{1+2^{k+1}%
\nu }\cdot u_{k}(s)^{-2^{k+1}\nu }+\frac{K_{0}}{2^{k}}u_{k+1}(s)\text{\ (}s>0%
\text{).}
\end{equation*}%
Since $w_{k}(s)\geq s^{(2^{k-1}-1)/(2^{k}\nu )}u_{k}(s)$ by definition (\ref%
{wkt}), we obtain%
\begin{eqnarray*}
u_{k+1}^{^{\prime }}(s) &\leq &-\frac{C_{N}^{-1}}{2^{k}}%
u_{k+1}(s)^{1+2^{k+1}\nu }\cdot \left\{ \frac{s^{(2^{k-1}-1)/(2^{k}\nu )}}{%
w_{k}(s)}\right\} ^{2^{k+1}\nu }+\frac{K_{0}}{2^{k}}u_{k+1}(s) \\
&=&-C_{N}^{-1}2^{-k}\frac{s^{2^{k}-2}}{w_{k}(s)^{2^{k+1}\nu }}%
u_{k+1}(s)^{1+2^{k+1}\nu }+K_{0}2^{-k}u_{k+1}(s).
\end{eqnarray*}%
Applying Proposition \ref{P9} in Appendix with $u(s)=u_{k+1}(s),\theta
=2^{k+1}\nu $ and $b=C_{N}^{-1}2^{-k},p=2^{k},w=w_{k},$ and $%
K=K_{0}2^{-k},a=1$, we obtain that%
\begin{eqnarray*}
u_{k+1}(s) &\leq &\left\{ \frac{2}{2^{k+1}\nu }\cdot \frac{2^{k}}{%
C_{N}^{-1}2^{-k}}\right\} ^{1/(2^{k+1}\nu )}s^{-(2^{k}-1)/(2^{k+1}\nu )}\exp
\left\{ K_{0}{ 4^{-k}}s\right\} w_{k}(s) \\
&=&\left\{ \frac{2^{k}}{C_{N}^{-1}\nu }\right\} ^{1/(2^{k+1}\nu
)}s^{-(2^{k}-1)/(2^{k+1}\nu )}\exp \left\{K_{0} { 4^{-k}}s\right\} w_{k}(s),
\end{eqnarray*}%
that is,%
\begin{equation*}
s^{(2^{k+1}-2)/(2^{k+2}\nu )}u_{k+1}(s)\leq \left\{ \frac{2^{k}}{%
C_{N}^{-1}\nu }\right\} ^{1/(2^{k+1}\nu )}\exp \left\{ K_{0}{ 4^{-k}}s\right\}
w_{k}(s)
\end{equation*}%
for any $s>0$.

Therefore, for any $t>0$,$k>0$,
\begin{eqnarray}
w_{k+1}(t) &=&\sup_{s\in (0,t]}\left\{ s^{(2^{k+1}-2)/(2^{k+2}\nu
)}u_{k+1}(s)\right\} \leq \left\{ \frac{2^{k}}{C_{N}^{-1}\nu }\right\}
^{1/(2^{k+1}\nu )}\sup_{s\in (0,t]}\left\{ \exp \left\{ K_{0}{ 4^{-k}}s\right\}
w_{k}(s)\right\}  \notag \\
&\leq &\left\{ \frac{2^{k}}{C_{N}^{-1}\nu }\right\} ^{1/(2^{k+1}\nu )}\exp
\left\{ K_{0}2^{-k}t\right\} w_{k}(t) { =:} (Da^{k})^{2^{-k}}w_{k}(t)  \label{16}
\end{eqnarray}%
where $D:=(C_{N}^{-1}\nu )^{-1/(2\nu )}\exp \left( K_{0}t\right) $ and $%
a:=2^{1/(2\nu )}>1$. { For the second inequality, we use $2^{-k}\leq1$.}

By iteration, we see from (\ref{16}), (\ref{15}) that for any $k\geq 1$%
\begin{eqnarray*}
w_{k+1}(t) &\leq &(Da^{k})^{2^{-k}}w_{k}(t)\leq (Da^{k})^{2^{-k}}\cdot
(Da^{k-1})^{2^{-(k-1)}}w_{k-1}(t)\leq \cdots \\
&\leq &D^{\frac{1}{2^{k}}+\frac{1}{2^{k-1}}+\cdots +\frac{1}{2}}a^{\frac{k}{%
2^{k}}+\frac{k-1}{2^{k-1}}+\cdots +\frac{1}{2}}w_{1}(t) \\
&\leq &\max \left\{ 1,D\right\} a^{2}w_{1}(t)\leq C_{1}\exp \left\{
K_{0}t\right\} w_{1}(t) \\
&\leq &C_{1}\exp \left( 2K_{0}t\right) \left\Vert f\right\Vert _{2},
\end{eqnarray*}%
where $C_{1}=\max \left\{ 1,(C_{N}^{-1}\nu )^{-1/(2\nu )}\right\} a^{2}$,
thus proving (\ref{17}).
\end{proof}

We now estimate the heat kernel $q_{t}^{(\rho )}(x,y)$ of the truncated
Dirichlet form $(\mathcal{E}^{(\rho )},\mathcal{F}^{(\rho )})$.

\begin{lemma}
\label{L1}Assume that conditions $\left( TJ\right) ,(DUE)$ hold. Then for
any ball $B$ of radius $r\in (0,R_{0})$ and for any $0<\rho \leq r$, the
heat kernel $q_{t}^{(\rho )}(x,y)$ of $(\mathcal{E}^{(\rho )},\mathcal{F}%
^{(\rho )})$ exists and satisfies%
\begin{equation}
q_{t}^{(\rho )}(x,y)=0  \label{22}
\end{equation}%
for $\mu $-almost all $x\in B,y\in B^{c}$ and for all $t>0$.
\end{lemma}

\begin{proof}
Let $0\leq f\in \mathcal{F}\cap L^{\infty}$ with $\left\Vert f\right\Vert _{2}=1$%
. By (\ref{17}), we have%
\begin{equation*}
t^{(2^{k-1}-1)/(2^{k}\nu )}\left\Vert f_{t}\right\Vert _{2^{k}}\leq
C_{1}\exp (2K_{0}t)
\end{equation*}%
which gives by letting $k\rightarrow \infty $ that%
\begin{equation}
\left\Vert Q_{t}^{\psi }f\right\Vert _{\infty }=\left\Vert f_{t}\right\Vert
_{\infty }\leq \frac{C_{1}}{t^{1/(2\nu )}}\exp (2K_{0}t)  \label{19}
\end{equation}%
for any $t>0$. It follows that%
\begin{equation*}
\left\Vert Q_{t}^{\psi }\right\Vert _{2\rightarrow \infty
}:=\sup_{\left\Vert f\right\Vert _{2}=1}\left\Vert Q_{t}^{\psi }f\right\Vert
_{\infty }\leq \frac{C_{1}}{t^{1/(2\nu )}}\exp (2K_{0}t).
\end{equation*}%
Since $Q_{t}^{-\psi }$ is the adjoint of the operator $Q_{t}^{\psi }$, we
see that
\begin{equation*}
\left\Vert Q_{t}^{\psi }\right\Vert _{1\rightarrow 2}:=\sup_{\left\Vert
f\right\Vert _{1}=1}\left\Vert Q_{t}^{\psi }f\right\Vert _{2}=\left\Vert
Q_{t}^{-\psi }\right\Vert _{2\rightarrow \infty }\leq \frac{C_{1}}{%
t^{1/(2\nu )}}\exp (2K_{0}t)
\end{equation*}%
and thus,
\begin{equation*}
\left\Vert Q_{t}^{\psi }\right\Vert _{1\rightarrow \infty }\leq \left\Vert
Q_{t/2}^{\psi }\right\Vert _{1\rightarrow 2}\left\Vert Q_{t/2}^{\psi
}\right\Vert _{2\rightarrow \infty }\leq \frac{C_{1}^{2}}{(t/2)^{1/\nu }}%
\exp (2K_{0}t).
\end{equation*}%
From this, we conclude that the heat kernel $q_{t}^{(\rho )}(x,y)$ of $(%
\mathcal{E}^{(\rho )},\mathcal{F}^{(\rho )})$ exists and satisfies
\begin{equation}
q_{t}^{(\rho )}(x,y)\leq \frac{C}{t^{1/\nu }}\exp \left( 2K_{0}t+\lambda (%
\mathbf{1}_{B}(y)-\mathbf{1}_{B}(x))\right) .  \label{20}
\end{equation}

Therefore, for $\mu $-almost all $x\in B,y\in B^{c}$ and for all $t>0$%
\begin{equation*}
q_{t}^{(\rho )}(x,y)\leq \frac{C}{t^{1/\nu }}\exp \left( 2K_{0}t-\lambda
\right) .
\end{equation*}%
Letting $\lambda \rightarrow \infty $, we have $q_{t}^{(\rho )}(x,y)=0$. The
proof is complete.
\end{proof}

Lemma \ref{L1} says that the heat kernel $q_{t}^{(\rho )}(x,y)$ of the
truncated Dirichlet form will \emph{vanish} for any $t>0$ when points $x$
and $y$ are separated by any ball of radius $r>\rho $. This means that the
Hunt process associated with the truncated Dirichlet form never goes far
away than the truncated range $\rho $. This lemma will give the strong tail
estimate of the heat semigroup $\left\{ P_{t}\right\} $ of the form $\left(
\mathcal{E},\mathcal{F}\right) $. This new phenomenon arises from the
ultra-metric property of the space $M$. { General metric spaces do not}
admit this nice property. For example, { for $\delta \in (0,1)$ ($\delta$ is similar to $\rho$ in this paper), the Dirichlet form $(\mathcal{E}^{\delta},\mathcal{F}^{\delta})$ is regular on $\mathbb{R}^d$ (see \cite[(2.14)-(2.16)]{ChenKimKumagai.2008.MA833}). Let $Z^{\delta}$ be a symmetric Markov process associated
with $(\mathcal{E}^{\delta},\mathcal{F}^{\delta})$.
The generator of $Z^{\delta}$ is
\begin{align}
  \mathcal{L}^{\delta}u(x)=\lim_{\varepsilon \downarrow 0}\int_{\{y\in \mathbb{R}^d:|y-x|>\varepsilon\}}(u(y)-u(x))J_{\delta}(x,y)dy.\notag
\end{align}
where $J_{\delta}(x,y)$ is a symmetric non-negative function (see \cite[(3.4)]{ChenKimKumagai.2008.MA833}). $Z^{\delta}$ has an quasi-continuous heat kernel $q^{(\delta)}_t(x,y)$ defined on
$[0,+\infty)\times \mathbb{R}^d \times \mathbb{R}^d$. For $\delta_0>0$, there exist $c=c(\delta_0)>0$ such that $q^{(\delta)}_t(x,y)\geq ct^{-d/2}$ for every $t>\delta_0$ and quasi everywhere $x,y$ with $|x-y|^2\leq t$ (see \cite[Theorem $3.4$]{ChenKimKumagai.2008.MA833})}.
\

\begin{proposition}
\label{P4}Assume that conditions $\left( TJ\right) ,\left( DUE\right) $
hold. Then for any ball $B$ of radius $r\in (0,R_{0})$ and for any $t>0$%
\begin{equation}
P_{t}\mathbf{1}_{B^{c}}\leq \frac{Ct}{r^{\beta }}\text{ in B}  \label{23}
\end{equation}%
for some universal constant $C>0$ (independent of $t,B$).
\end{proposition}

\begin{proof}
By Lemma \ref{L1}, we have that for $\mu $-almost all $x\in B$ and all $t>0$%
,
\begin{equation*}
Q_{t}\mathbf{1}_{B^{c}}(x)=\int_{B^{c}}q_{t}^{(\rho )}(x,y)d\mu (y)=0
\end{equation*}%
if $0<\rho \leq r$. Applying Proposition \ref{P8} in Appendix for $\rho =%
\frac{r}{2},\Omega =M,f=\mathbf{1}_{B^{c}}$, we see that%
\begin{equation*}
P_{t}\mathbf{1}_{B^{c}}\leq Q_{t}\mathbf{1}_{B^{c}}+2t\sup_{x\in
M}J(x,B(x,\rho )^{c})\leq C\frac{t}{r^{\beta }}\text{ \ in }B.
\end{equation*}%
The proof is complete.
\end{proof}

We are in a position to prove Theorem \ref{T1}.

\begin{proof}[Proof of Theorem \protect\ref{T1}]
{  Fix points $x_{0},y_{0}\in M$ and fix $t\in(0,R_{0}^{\beta })$. Without loss of generality, assume that $\frac{d(x_0,y_0)}{t^{1/\beta }}\geq 1$; otherwise $\left( wUE\right) $ follows
directly from $\left( DUE\right) $ and nothing is proved.
Let $r=\frac{d(x_0,y_0)}{2}$. Noting that $P_t1_{B^c}$ is monotone decreasing in $r$, we have from (\ref{23}), for $r>0, t>0$
\begin{equation}\label{24}
P_t1_{B^c}\leq\frac{Ct}{(r\wedge R_0 )^{\beta}}
\end{equation}}
It follows from condition $\left( DUE\right) $ and (\ref%
{24}), for $\mu $-almost all $x\in B(x_{0},r),y\in B(y_{0},r)$, we have
\begin{eqnarray*}
p_{2t}(x,y) &=&\int_{M}p_{t}(x,z)p_{t}(z,y)d\mu (z) \\
&\leq &\int_{B(x_{0},r)^{c}}p_{t}(x,z)p_{t}(z,y)d\mu
(z)+\int_{B(y_{0},r)^{c}}p_{t}(x,z)p_{t}(z,y)d\mu (z) \\
&\leq &\sup_{z\in M}p_{t}(z,y)\int_{B(x_{0},r)^{c}}p_{t}(x,z)d\mu
(z)+\sup_{z\in M}p_{t}(x,z)\int_{B(y_{0},r)^{c}}p_{t}(z,y)d\mu (z) \\
&\leq &2\frac{C}{t^{\alpha /\beta }}\cdot \frac{Ct}{(r\wedge R_0)^{\beta }}.
\end{eqnarray*}%
Therefore for $\mu $-almost all $x_{0},y_{0}\in M$ and $t\in (0,R_{0}^{\beta
})$, we conclude that%
\begin{equation*}
p_{2t}(x_{0},y_{0})\leq \frac{C}{t^{\alpha /\beta }}\cdot \frac{t}{
(d(x_{0},y_{0})\wedge R_0)^{\beta }}.
\end{equation*}%
{ This inequality together with $\left( DUE\right) $ will imply $\left( wUE\right) $}. The proof is complete.
\end{proof}

\section{Appendix}

The following was first shown in \cite[Lemma 3.21]%
{CarlenKusuokaStroock.1987.AIHPPS245}, and then was modified in \cite[Lemma
2.6]{MuruganSaloff.2015.dav} (see also \cite[Lemma 3.4]{HuLi.2018}).

\begin{proposition}
\label{P9}Let $w:(0,\infty )\rightarrow (0,\infty )$ be a non-decreasing
function and suppose that $u\in C^{1}([0,\infty );(0,\infty ))$ satisfies
that for all $t\geq 0$,
\begin{equation*}
u^{\prime }(t)\leq -b\frac{t^{p-2}}{w^{\theta }(t)}u^{1+\theta }(t)+Ku(t)
\end{equation*}%
for some $b>0,$ $p>1$, $\theta >0$ and $K>0$. Then
\begin{equation*}
u(t)\leq \left( \frac{2p^{a}}{\theta b}\right) ^{1/\theta }t^{-(p-1)/\theta
}e^{Kp^{-a}t}w(t)
\end{equation*}%
for any $a\geq 1$.
\end{proposition}

The following result was proved in \cite[Proposition 4.6]%
{GrigoryanHuLau.2014.TAMS6397}.

\begin{proposition}
\label{P8}For an open set $\Omega \subset M$, let $\left\{ P_{t}^{\Omega
}\right\} $ and $\left\{ Q_{t}^{\Omega }\right\} $ be the heat semigroups of
$(\mathcal{E},$\QTR{cal}{F}$(\Omega ))$ and $(\mathcal{E}^{(\rho )},$%
\QTR{cal}{F}$^{(\rho )}(\Omega ))$, respectively. Then for any $t>0$
\begin{equation*}
P_{t}^{\Omega }f\leq Q_{t}^{\Omega }f+2t\sup_{x\in M}J(x,B(x,\rho
)^{c})\left\Vert f\right\Vert _{\infty }.
\end{equation*}
\end{proposition}

\bibliographystyle{siam}
\bibliography{reflong}

\end{document}